\documentclass[12pt, reqno]{amsart}
\usepackage{amsmath, amsthm, amscd, amsfonts, amssymb, graphicx, color}
\usepackage[bookmarksnumbered, colorlinks, plainpages]{hyperref}
\hypersetup{colorlinks=true,linkcolor=red, anchorcolor=green, citecolor=cyan, urlcolor=red, filecolor=magenta, pdftoolbar=true}

\newtheorem{theorem}{Theorem}[section]
\newtheorem{lemma}[theorem]{Lemma}
\newtheorem{proposition}[theorem]{Proposition}
\newtheorem{corollary}[theorem]{Corollary}

\newtheorem*{TheoremA}{Main Theorem}

\theoremstyle{definition}

\theoremstyle{remark}

\numberwithin{equation}{section}

\newcommand\C{\mathbb C}
\newcommand\D{\mathbb D}

\newcommand\rimu{R^\infty(K,\mu)}
\newcommand\rikmu{R^\infty(K,\mu)}

\newcommand\limu{L^\infty(\mu)}
\renewcommand\i{\infty}
\newcommand\area{{\frak m}}

\begin{document}

\setcounter{page}{1}

\title[Spectral Mapping Theorems]{A Note on Spectral Mapping Theorems for Subnormal Operators}

\author[L. Yang]{Liming Yang$^1$}

\address{$^1$Department of Mathematics, Virginia Polytechnic and State University, Blacksburg, VA 24061.}
\email{\textcolor[rgb]{0.00,0.00,0.84}{yliming@vt.edu}}




\begin{abstract} For a compact subset $K\subset \C$ and a positive finite Borel measure $\mu$ supported on $K,$ let $\text{Rat}(K)$ denote the space of rational functions with poles off $K,$ let $R^\i (K,\mu)$ be the weak-star closure  of $\text{Rat}(K)$ in $L^\i(\mu),$ and let $R^2 (K,\mu)$ be the closure  of $\text{Rat}(K)$ in $L^2(\mu).$  
We show that there exists a compact subset $K\subset \C,$ a positive finite Borel measure $\mu$ supported on $K,$ and a function $f\in R^\i (K,\mu)$ such that $R^\i (K,\mu)$ has no non-trivial direct $L^\i$ summands, $f$ is invertible in  $R^2 (K,\mu)\cap L^\i(\mu),$ and $f$ is not invertible in $R^\i (K,\mu).$   The result answers an open question concerning spectral mapping theorems for subnormal operators raised by J. Dudziak \cite{dud84} in 1984.
\end{abstract} \maketitle

\section{\textbf{Introduction}}

For a Borel subset $B$ of the complex plane $\C,$ let $M_0(B)$ denote the set of finite complex-valued Borel measures that are compactly supported in $B$ and let $M_0^+(B)$ be the set of positive measures in $M_0(B).$ The support of $\nu\in M_0(\C),$ $\text{spt}(\nu),$ is the smallest closed set that has full $|\nu|$ measure. For a Borel set $A\subset \C,$ $\nu_A$ denotes $\nu$ restricted to $A.$

For a compact subset $K\subset \C$ and $\mu\in M_0^+(K),$ the functions in 
$\mbox{Rat}(K) := \{q:\mbox{$q$ is a rational function 
with poles off $K$}\}$
are members of $L^\i (\mu)$. We let $\rikmu$ be the weak-star closure of $\mbox{Rat}(K)$ in $\limu.$ 
There exists a Borel partition $\{\Delta_0,\Delta_1\}$ of $K$ such that 
\begin{eqnarray}\label{RIDecomp}
\ \rikmu = L^\i (\mu|_{\Delta_0}) \oplus R^\i (K,\mu|_{\Delta_1}) 
\end{eqnarray}
where $R^\i (K,\mu|_{\Delta_1})$ contains no non-trivial $L^\i$ summands (see \cite[Proposition 1.16 on page 281]{conway} ).
Call $\rikmu$ pure if $\Delta_0 = \emptyset$ in \eqref{RIDecomp}. Because of \eqref{RIDecomp}, we shall assume that $\rikmu$ is pure.
The {\em envelope} $E(K,\mu)$ with respect to $K$ and $\mu\in M_0^+(K)$ is the set of points $\lambda \in  K$ such that there exists $\mu_\lambda\in M_0(K)$ that is absolutely continuous with respect to $\mu$ such that $\mu_\lambda(\{\lambda\}) = 0$ and $\int f(z) d\mu_\lambda(z) = f(\lambda)$ for each $f \in \text{Rat}(K).$ From Proposition \ref{EProp} (b), $E(K,\mu)$ is a nonempty $\area$ measurable set with area density one at each of its points, where $\area$ is the area (Lebesgue) measure on $\C$.
For $\lambda\in E(K,\mu)$ and $f\in R^\i (K,,\mu),$ set $\rho_{K,\mu}(f)(\lambda) = \int fd\mu_\lambda.$ Clearly $\rho_{K,\mu}(f)(\lambda)$ is independent of the particular $\mu_\lambda$ chosen. 
We thus have a map $\rho_{K,\mu}$ called Chaumat's map for $K$ and $\mu,$ which associates
to each function in $\rikmu$ a point function on $E(K,\mu).$ Chaumat's Theorem \cite{cha74} (also see \cite[page 288]{conway}) states the following: The map $\rho_{K,\mu}$ is an isometric isomorphism and a weak-star homeomorphism from $\rikmu$ onto $R^\i (\overline {E(K,\mu)}, \area_{E(K,\mu)}),$ where $\overline {B}$ denotes the closure of the set $B\subset \C$. 

For $\mathcal H$ a complex separable Hilbert space, $\mathcal L(\mathcal H)$ denotes the space of all bounded linear operators on $\mathcal H.$ The spectrum of an operator $T\in \mathcal L(\mathcal H)$ is denoted $\sigma(T).$ An operator $S\in \mathcal L(\mathcal H)$ is called {\em  subnormal} if there exists a complex separable Hilbert space $\mathcal K$ containing $\mathcal H$ and a normal operator $N\in \mathcal L(\mathcal K)$ such that $N\mathcal H \subset \mathcal H$ and $S = N|_{\mathcal H}.$ Such an $N$ is called a {\em minimal normal extension} (mne) of $S$ if the smallest closed subspace of $\mathcal K$ containing $\mathcal H$ and reducing $N$ is $\mathcal K$ itself. Any two mnes of $S$ are unitarily equivalent in a manner that fixes $S.$ 
The spectrum $\sigma (S)$ is the union of $\sigma (N)$ and some collection of bounded components of $\C\setminus \sigma (N).$ The book \cite{conway} is a good reference for basic information for subnormal operators.

For a subnormal operator $S\in \mathcal L(\mathcal H)$ with $N=mne(S),$ let $\mu$ be the scalar-valued spectral measure (svsm) for $N.$ Since $\text{spt}(\mu) = \sigma(N) \subset  \sigma (S),$ we see that $\text{Rat}(\sigma (S)) \subset L^\i(\mu).$
Thus, for $f\in R^\i(\sigma (S), \mu),$ the normal operator $f(N)$ is well defined and $f(N)\mathcal H \subset \mathcal H.$
Therefore, the operator $f(S) := f(N) |_{\mathcal H}$ defines a functional calculus for $S$ and $f\in R^\i(\sigma (S), \mu).$ 
 The operator $S$ is pure if $S$ does not contain a non-trivial normal summand and $S$ is $R^\i$ pure if $R^\i(\sigma (S), \mu)$ is pure. Clearly, if $S$ is pure, then $S$ is $R^\i$ pure.

The following theorem in \cite[Theorem 1.1 and Corollary 1.2]{y23} proves the conjecture $\diamond$ posed by J. Dudziak in \cite[on page 386]{dud84}.

\begin{theorem}\label{InvTheorem} (Yang 2023)
Let $S\in \mathcal L(\mathcal H)$ be a subnormal operator and let $\mu$ be the scalar-valued spectral measure for $N=mne(S).$ Assume that $S$ is $R^\i$ pure. If for $f\in R^\i(\sigma (S), \mu),$ there exists $\epsilon_f > 0$ such that 
\[
\ |\rho_{\sigma (S), \mu}(f)(z)| \ge \epsilon_f,~ \area_{E(\sigma (S), \mu)}-a.a.,
\] 
then $f$ is invertible in $R^\i(\sigma (S), \mu)$ and $f(S)$ is invertible.	Consequently, $\sigma(f(S)) \subset f(\sigma(S)),$ where $f(\sigma(S))$ denotes the essential range of $\rho_{\sigma (S), \mu}(f)$ in $L^\i (\area_{E(\sigma (S), \mu)}).$
\end{theorem}

The second open question posted in \cite[page 386]{dud84} is that of whether the inclusion  $f(\sigma(S)) \subset \sigma(f(S))$ holds for every $R^\i$ pure subnormal operator $S$ and every $f\in R^\i(\sigma (S), \mu).$ J. Dudziak conjectured that the answer to the  question would be no and suggested that an appropriate place to look for an example would be among rationally cyclic $R^\i$ pure subnormal operators. For a compact subset $K\subset \C$ and $\mu \in M_0^+(K),$ let $R^2(K,\mu)$ denote the closure of $\text{Rat}(K)$ in $L^2(\mu)$ and let $S_\mu$ denote the multiplication by $z$ on $R^2(K,\mu).$ Clearly, $R^2(K,\mu) = R^2(\sigma(S_\mu),\mu).$ So we will always assume $K = \sigma(S_\mu).$ It is well known that a rationally cyclic subnormal operator is unitarily equivalent to $S_\mu$ on $R^2(\sigma(S_\mu),\mu)$ (see \cite[III.5.2]{c81}). Our main theorem provides such an example.

\begin{TheoremA}
There exists $\mu \in M_0^+(\C)$ and a function $f \in R^\i(K,\mu),$ where $K = \text{spt}\mu,$ such that $R^\i(K,\mu)$ is pure,
$|f| \ge \epsilon,~\mu-a.a.$ for some $\epsilon > 0,$ $f$ is invertible in $R^2(K,\mu) \cap L^\i (\mu),$ and $f$ is not invertible in $R^\i (K,\mu).$
\end{TheoremA}

As a result, our corollary below answers the above open question negatively.

\begin{corollary}
There exists $\mu \in M_0^+(\C)$ and there exists a function $f \in R^\i(\sigma(S_\mu),\mu)$ such that $S_\mu$ is $R^\i$ pure and $f(\sigma(S_\mu)) \setminus \sigma(f(S_\mu)) \ne \emptyset.$	
\end{corollary}

\begin{proof}
Let $\mu$ and $f$ be as in Main Theorem. Then, by Main Theorem, there exists $f_0\in R^2(K,\mu) \cap L^\i (\mu)$ such that $ff_0 = 1$ and $f$ is not invertible in $R^\i (K,\mu).$ Since $f_0 R^2(K,\mu) \subset R^2(K,\mu),$ we see that $f_0(S_\mu)f(S_\mu) = f(S_\mu)f_0(S_\mu) = I$ which implies $0\notin \sigma(f(S_\mu)).$ Using Theorem \ref{InvTheorem}, we conclude that zero belongs to the essential range of $\rho_{\sigma(S_\mu), \mu} (f)$ in $L^\i (\area_{E(\sigma (S_\mu), \mu)}).$ Therefore, $0\in f(\sigma(S_\mu)).$
	\end{proof}

\section{\textbf{Proof of Main Theorem}}

The Cauchy transform of $\nu \in M_0(\C)$ is defined by
\[
\ \mathcal C\nu (z) = \int \dfrac{1}{w - z} d\nu (w)
\]
for all $z\in\mathbb {C}$ for which
$\int \frac{d|\nu|(w)}{|w-z|} < \infty .$ A standard application of Fubini's
Theorem shows that $\mathcal C\nu \in L^s_{loc}(\mathbb {C} )$ for $ 0 < s < 2.$ In particular, it is
defined for $\area-a.a..$

The elementary properties of $E(K,\mu)$ are listed below.

\begin{proposition}\label{EProp} Let $\mu \in M_0^+(K)$ for some compact subset $K\subset \C.$ If $\rimu$ is pure, then the following properties hold.
\newline
(a) $E(K,\mu)$ is the set of weak-star continuous homomorphisms on $\rikmu$ (see \cite[Proposition VI.2.5]{conway}).
\newline
(b) $E(K,\mu)$ is a nonempty $\area$ measurable set with area density one at each of its points (see \cite[Proposition VI.2.8]{conway}).
\end{proposition}

For $\lambda\in \C$ and $\delta > 0,$ let $\D(\lambda,\delta) = \{z:~ |z - \lambda | < \delta\}.$ Set $\D = \D(0,1).$
We first construct a compact subset $K$ satisfying the following conditions.
\newline
(1) There exists a sequence of disjoint closed disks $\overline{\D(z_n,r_n)} \subset \D(0,\frac 12)$ such that $K_1 := \overline{\D(0,\frac 12)} \setminus \cup_{n=1}^\i \D(z_n,r_n)$ has no interior and $0\in K_1.$
\newline
(2) There exist $\delta_n$ and $r_n$ with $0 < \frac{r_n}{2} \le \delta_n < r_n$ such that $K = K_1 \cup \cup_{n=0}^\i A_n,$ where $A_n = \{z:~\delta_n \le |z - z_n| \le r_n\},$ $z_0 = 0,$ $\delta_0 = \frac 12,$ and $r_0 = 1.$
\newline
(3) 
\[
\ \sum_{m=1}^\i 2^m\sum_{2^{-m-1} \le |z_n| < 2^{-m}}r_n < \i.
\]
\newline
(4) There exists a sequence $\{f_n\}\subset \text{Rat}(K)$ such that $f_n(z)$ is analytic on $\D\setminus \cup_{k=1}^n \overline {\D(z_k,\delta_k)},$ $f_n(0) = 0,$ $\frac 12 < |f_n(z)| < 2$ for $z\in \cup_{k=0}^n A_k.$ 
\newline
(5) The envelop $E(K, \area_K)$ satisfies $\area(K\setminus E(K, \area_K)) = 0.$
\newline
(6) $0\in E(K, \area_K).$ Therefore, $\phi_0(b) = \rho_{K, \area_K}(b)(0) = \int bg_0d\area_K$ for $b\in R^\i(K, \area_K)$ is a weak-star continuous homomorphism on $R^\i(K, \area_K).$
\newline
(7) There exists $f\in R^\i(K, \area_K)$ such that $\|f\| \le 2,$ $|f(z)| \ge \frac 12$ for $z\in \cup_{n=0}^\i A_n,$ and $\phi_0(f) = 0.$ Hence, $f$ is not invertible in $R^\i(K, \area_K).$

\begin{proof}
Let $\{\lambda_k\}\subset \D(0,\frac 12)\setminus \{0\}$ be a dense subset of $\D(0,\frac 12),$ $|\lambda_1 | > \frac 13,$ and $z_1 = \lambda_{k_1},$ where $k_1 = 1.$ Let
$f_1(z) = \frac 32 z.$ We choose $r_1, \delta_1$ small enough such that $\frac {r_1}{2} \le\delta_1 < r_1 < \frac 12\min(\frac 12 - |z_1|, |z_1| - \frac 13).$ Then $\frac 12 < |f_1(z)| < 2$ for $z \in A_0 \cup A_1.$ 

Suppose that we have constructed $f_1,f_2,...,f_n$ with $z_n = \lambda_{k_n}$ satisfying $\lambda_j \in \cup_{k=1}^n \overline {\D(z_k,r_k)}$ for $1 \le j \le k_n$ and the condition (4). We now construct $f_{n+1}$ as the following. We find $k_{n+1} > k_n$ such that $\lambda_{k_n+1},...,\lambda_{k_{n+1}-1} \in \cup_{k=1}^n \overline {\D(z_k,r_k)}$ and $\lambda_{k_{n+1}} \notin \cup_{k=1}^n \overline {\D(z_k,r_k)}.$ Set $z_{n+1} = \lambda_{k_{n+1}}.$ Let $m$ be the integer such that $2^{-m-1} \le |z_{n+1}| < 2^{-m}.$ Let $a_{n+1} = f_n(z_{n+1})$ and let $\frac {r_{n+1}}{2} \le\delta_{n+1} <r_{n+1} < 2^{-n-m-1}$ be chosen below.

If $|a_{n+1}| > \frac 12,$ then set $f_{n+1}(z) = f_n(z).$ We choose $r_{n+1}$ small enough such that $|f_{n+1}(z)| > \frac 12$ for $z\in \overline{\D(z_{n+1},r_{n+1})}$ and set $\delta_{n+1} = \frac{r_{n+1}}{2}.$

Now we assume that $|a_{n+1}| \le  \frac 12.$ Let 
\[
\ b_{n+1}(z) = \dfrac 54 \dfrac{z}{|z_{n+1}|} \frac{r_{n+1}}{z-z_{n+1}}.
\] 
Let $min_n = \min_{z\in \cup_{k=0}^n A_k} |f_n(z)|$ and $max_n = \max_{z\in \cup_{k=0}^n A_k} |f_n(z)|.$ Then $\frac 12 < min_n < max_n < 2.$ Choose $r_{n+1}$ and $\delta_{n+1}$ small enough such that 
\[
\ \max_{z\in \cup_{k=0}^n A_k} |b_{n+1}(z)| < \min (min_n - \frac 12, 2-max_n)
\]
and $1 < |b_{n+1}(z)| < \frac 32$ for $z\in A_{n+1}.$
Then $f_{n+1}(z) = f_n(z) + b_{n+1}(z)$ satisfies the condition (4).
It is easy to verify (1)-(3). 

Let $\mu$ be the arc length measure of $\partial_e K = \partial \D \cup \cup _{n=1}^\i \partial \D(z_n, \delta_n).$ Let $g = \frac{1}{2\pi i}\frac{dz}{d\mu}.$ Then $|g(z)| =  \frac{1}{2\pi}~\mu-a.a.$ and $\int bgd\mu = 0$ for $b \in R^\i (K,\mu).$ Thus, $R^\i (K,\mu)$ is pure.
For $\lambda \in K \setminus \partial_e K$ with $\int \frac{|g(z)|}{|z - \lambda|} d \mu(z) < \i$ and $h\in \text{Rat}(K),$ we have 
\begin{eqnarray}\label{Eq0}
\ h(\lambda) = \int_{\partial \D \cup \cup _{n=1}^m \partial \D(z_n, \delta_n)} \dfrac{h(z)g(z)}{z - \lambda} d \mu(z) \rightarrow \int h(z)g_\lambda (z) d\mu(z) \text{ as }m\rightarrow \i, 
\end{eqnarray}
where $g_\lambda (z) = \frac{g(z)}{z - \lambda}.$ Hence, $\lambda\in E(K,\mu),$ which implies $K = \overline {E(K,\mu)}$ and $\area (K \setminus E(K,\mu)) = 0$ since $\int \frac{|g(z)|}{|z - \lambda|} d \mu(z) < \i~\area-a.a..$ 

Using Chaumat's theorem (\cite{cha74} or \cite[Chaumat's Theorem on page 288]{conway}),
the Chaumat's mapping $\rho_{K,\mu}$ is an isometric isomorphism and a weak$^*$ homeomorphism from $\rikmu$ onto $R^\i (\overline {E(K,\mu)}, \area_{E(K,\mu)}) = R^\i (K, \area_K).$ Hence, $E(K,\area_K) = E(K,\mu)$ by Proposition \ref{EProp} (a).
(5) is proved.

Using (3), we have
$\int \frac{|g(z)|}{|z|} d \mu(z) < \i$ Hence, $0\in E(K,\area_K).$ (6) follows from Proposition \ref{EProp} (a).

Since $\|f_n\| \le 2$ in (4), we can choose a sequence $\{f_{n_j}\}$ such that $f_{n_j}$ converges to $f\in R^\i (K, \area_K)$ in weak-star topology. It is straightforward to verify the condition (7) for $f.$
\end{proof}

\begin{lemma} \label{ExampleLemma}
Let $K$ be a compact subset satisfying (1)-(7). Then there exists a continuous function $W(z)$ on $\C$ satisfying $W(z) > 0$ for $z\in \cup_{n=0}^\i \text{int}(A_n)$ and $W(z) = 0$ for $z\in \C \setminus \cup_{n=0}^\i \text{int}(A_n)$ such that the following properties hold.
\newline
(a) The following decomposition hold:
\[
\ R^2(K, W\area) = \bigoplus_{n=0}^\i R^2(A_n, W\area_{A_n}). 
\]
\newline
(b) $R^\i(K, \area_K) \subset R^\i(K, W\area),$ $\area (K \setminus E(K, W\area)) = 0,$ and $0\in E(K, W\area).$ 
\newline
(c) Let $f$ be the function constructed as in (7). Then $f$ is invertible in $R^2(K, W\area) \cap L^\i(W\area),$ but $f$ is not invertible in $R^\i(K, W\area).$
\end{lemma}

\begin{proof}
Define
\[
\ W(z) = \begin{cases} (r_n - |z - z_n|)^4, & \dfrac{\delta_n + r_n}{2} < |z - z_n| \le r_n\\ (|z - z_n| - \delta_n)^4, & \delta_n  \le |z - z_n| \le \dfrac{\delta_n + r_n}{2}\\ 0, & z\in \C \setminus \cup_{n=0}^\i A_n.  \end{cases}
\]
Then $W(z)$ is continuous on $\C$ such that $W(z) > 0$ for $z\in \cup_{n=0}^\i \text{int}(A_n)$ and $W(z) = 0$ for $z\in \C \setminus \cup_{n=0}^\i \text{int}(A_n).$ Moreover, for $\lambda \notin A_n$ and $z\in A_n,$ we have
\[
\ W(z) \le |z-\lambda|^4\text{ and }W(z) \le |z-z_n|^4.	
\]
For $\lambda \in \overline{\D(0,\frac 12)} \setminus \cup_{n=1}^\i \D(z_n,r_n)$ and $g \in R^2(K, W\area)^\perp := \{g\in L^2(W\area):~ \int h(z)g(z) d\mu = 0 \text{ for }h \in \text{Rat}(K)\},$
\begin{eqnarray}\label{Eq1}
\ \int \dfrac{|g(z)|}{|z-\lambda|}Wd\area \le \left (\int \dfrac{1}{|z-\lambda|^2}Wd\area \right )^{\frac 12} \|g\| \le \sqrt\pi \|g\|.
\end{eqnarray}
Hence, $\mathcal C(gW\area)(\lambda)$ is well defined. Also we have
\[
\begin{aligned}
\ & \left | \mathcal C(gW\area)(\lambda) -   \mathcal C(gW\area)(z_n)\right | \\
\ \le &	|z_n - \lambda| \left (\int \dfrac{1}{|z-z_n|^4}Wd\area \right )^{\frac 14} \left (\int \dfrac{1}{|z-\lambda|^4}Wd\area \right )^{\frac 14} \|g\| \\
\ \le &	\sqrt \pi |z_n - \lambda| \|g\|.
\end{aligned}
\]
By construction, we select a subsequence $\{z_{n_k}\}$ such that $z_{n_k}$ tends to $\lambda,$ which implies $\mathcal C(gW\area)(\lambda) = 0$ since $\mathcal C(gW\area)(z_n) = 0.$ Therefore, 
\begin{eqnarray}\label{Eq2}
\ \mathcal C(gW\area)(\lambda) = 0, ~ \lambda \in \overline{\D(0,\frac 12)} \setminus \cup_{n=1}^\i \D(z_n,r_n).
\end{eqnarray}
Thus, using \eqref{Eq1}, Fubini's Theorem, and \eqref{Eq2}, we get for $n \ge 1,$
\[
\begin{aligned}
\ \int \chi_{A_n}g Wd\area = & \dfrac{1}{2\pi i} \int \int_{\partial \D(z_n,r_n)} \dfrac{dz}{z - w} gWd\area(w) \\ 
\ = & - \dfrac{1}{2\pi i}\int_{\partial \D(z_n,r_n)} \mathcal C(gW\area)(z)dz \\
\ = & 0,
\end{aligned}
\]
where $\chi_{A_n}$ denotes the characteristic function of $A_n.$ Hence, $\chi_{A_n}\in R^2(K, W\area)$ for $n \ge 1.$ Thus, $\chi_{A_0}\in R^2(K, W\area)$ since $\chi_{A_0} = 1 - \sum_{n=1}^\i \chi_{A_n},~W\area-a.a..$ It is clear that $R^2(K, W\area_{A_n}) = R^2(A_n, W\area_{A_n}).$ (a) is proved.

For (b): we have, 
\[
\ R^\i(K, \area_K) \subset R^\i(K, W\area) \text{ and } \cup_{n=0}^\i \text{int}(A_n) \subset E(K, W\area). 
\]
Let $T=\partial\D(0,\frac{3}{4})\cup \cup_{n=1}^\i \partial \D(z_n,\frac{\delta_n+r_n}{2}).$ Using the same argument as in \eqref{Eq0},  for $\lambda \in \overline{\D(0,\frac{1}{2})}\setminus \cup_{n=1}^\i  \D(z_n,r_n)$ with $\int_T \frac{1}{|z - \lambda|} |dz| < \i,$ we conclude that 
\[
\ \phi_\lambda (r) = \dfrac{1}{2\pi i}\int_T \dfrac{r(z)}{z - \lambda} d z = r(\lambda) \text{ for } r\in\text{Rat}(K). 
\]
For $h\in R^\i(K, W\area),$ let $\{h_n\}\subset \text{Rat}(K)$  such that $h_n\rightarrow h$ in $L^\i(W\area)$ weak-star topology. Then there exists $C_1 > 0$ such that
$\|h_n\|_{L^\i(W\area)} \le C_1,$ $h_n(z)\rightarrow h(z)$ uniformly on any compact subset of $\cup_{n=0}^\i \text{int}(A_n),$ and $h$ is analytic on $\cup_{n=0}^\i \text{int}(A_n).$ Hence, $h_n(z)\rightarrow h(z),~z\in T,~ |dz|-a.a..$ Using the Lebesgue dominated convergence theorem, we have
\[
\ \dfrac{1}{2\pi i}\int_T \dfrac{h_n(z)}{z - \lambda} d z = h_n(\lambda) \rightarrow \dfrac{1}{2\pi i}\int_T \dfrac{h(z)}{z - \lambda} d z.
\]
Therefore, $\phi_\lambda$ extends a weak-star continuous homomorphism on $R^\i(K, W\area),$ which implies $\lambda \in E(K, W\area)$ by Proposition \ref{EProp} (a). Thus, $\area (K \setminus E(K, W\area)) = 0$ since  $\int_T \frac{1}{|z - \lambda|} |dz| < \i~\area-a.a..$ $0 \in E(K, W\area)$ follows from (3).

(c): Clearly,
\[
\ R^2(A_n, W\area_{A_n}) \cap L^\i (W\area_{A_n}) = R^\i (A_n, \area_{A_n}) = H^\i (\text{int}(A_n)),
\]
where $H^\i (\text{int}(A_n))$ is the algebra of bounded and analytic functions on $\text{int}(A_n).$
Since $|f(z)| \ge \frac 12, ~ z\in \cup_{n=0}^\i A_n,$ the function $f|_{A_n}$ is invertible in $R^2(A_n, W\area_{A_n}) \cap L^\i (W\area_{A_n})$ and $\|(f|_{A_n})^{-1}\| \le 2.$ Define
$h = \sum _{n=0}^\i (f|_{A_n})^{-1}.$
Then $\|h\| \le 2,$ $h \in R^2(K, W\area) \cap L^\i (W\area),$ and $fh = 1,~W\area-a.a..$ Therefore, 
 we conclude that $f$ is invertible in $R^2(K, W\area) \cap L^\i (W\area)$ and $f$ is not invertible in $R^\i(K, W\area)$ by $0 \in E(K, W\area)$ and (7). This proves (c).
\end{proof}

The proof of Main Theorem follows from Lemma \ref{ExampleLemma}.

\bigskip

{\bf Acknowledgments.} 
The author would like to thank Professor John M\raise.45ex\hbox{c}Carthy for carefully reading through the manuscript and providing many useful comments.

\bigskip

\bibliographystyle{amsplain}

\end{document}